\pdfoutput=1
\documentclass[letterpaper, 10 pt, conference]{ieeeconf}  % Comment this line out if you need a4paper

\IEEEoverridecommandlockouts                              % This command is only needed if 
\usepackage{amsmath}
\usepackage{color}
\usepackage{amssymb}  % assumes amsmath package installed
\usepackage{amsfonts}
\usepackage{graphicx}
\usepackage{dsfont}
\usepackage{psfrag}
\usepackage{mathtools}
\usepackage[noadjust]{cite}
\usepackage{bbm}
\usepackage{algorithm}
\usepackage{algorithmic}
\usepackage{savesym}
\usepackage{amsthm}
\usepackage[OT2,T1]{fontenc}
\DeclareSymbolFont{cyrletters}{OT2}{wncyr}{m}{n}
\DeclareMathSymbol{\Sha}{\mathalpha}{cyrletters}{"58}

\theoremstyle{definition}
\newtheorem{definition}{Definition}

\newtheorem{remark}{Remark}

\theoremstyle{plain}
\newtheorem{theorem}{Theorem}

\newtheorem{lemma}{Lemma}
\newtheorem{probstat}{Problem}

\DeclareMathOperator*{\argmin}{argmin}

\title{\LARGE \bf
Dealing with State Estimation in Fractional-Order Systems under Artifacts
}

\author{Sarthak Chatterjee$^{\dagger}$ and S\'{e}rgio Pequito$^{\ddagger}$% <-this % stops a space
\thanks{$^{\dagger}$Sarthak Chatterjee is with the department of Electrical, Computer, and Systems Engineering, Rensselaer Polytechnic Institute, Troy, NY 12180, USA
        {\tt\small chatts3@rpi.edu}}%
\thanks{$^{\ddagger}$S\'{e}rgio Pequito is with the Department of Industrial and Systems Engineering, Rensselaer Polytechnic Institute, Troy, NY 12180, USA
        {\tt\small goncas@rpi.edu}}%
}

\begin{document}

\maketitle
\thispagestyle{empty}
\pagestyle{empty}

%%%%%%%%%%%%%%%%%%%%%%%%%%%%%%%%%%%%%%%%%%%%%%%%%%%%%%%%%%%%%%%%%%%%%%%%%%%%%%%%
\begin{abstract}

Fractional-order dynamical systems are used to describe processes that exhibit long-term memory with power-law dependence. Notable examples include complex neurophysiological signals such as electroencephalogram (EEG) and blood-oxygen-level dependent (BOLD) signals. When analyzing different neurophysiological signals and other signals with different origin (for example, biological systems), we often find the presence of artifacts, that is, recorded activity that is due to external causes and does not have its origins in the system of interest. In this paper, we consider the problem of estimating the states of a discrete-time fractional-order dynamical system when there are artifacts present in some of the sensor measurements. Specifically, we provide necessary and sufficient conditions that ensure we can retrieve the system states even in the presence of artifacts. We provide a state estimation algorithm that can estimate the states of the system in the presence of artifacts. Finally, we present illustrative examples of our main results using real EEG data.

%Fractional-order dynamical systems are used to describe processes that exhibit long-term memory and power-law dependence of trajectories. Notable examples are complex neurophysiological signals such as electroencephalogram (EEG) and blood-oxygen-level dependent (BOLD) imaging. When analyzing EEG signals, we often find the presence of recorded activity that is due to external causes and does not have its origins in the brain. We call these artifacts, and they provide a systematic means for evaluating the EEG signals, and minimizing the scope for misinterpreting the same. In this paper, we consider the problem of estimating the states of a discrete-time fractional-order dynamical system when there are artifacts present in some of the sensor measurements. These artifacts are assumed to be arbitrary real numbers, and do not follow any given model. We provide necessary and sufficient conditions that guarantee when the states of the system can be estimated even in the presence of artifacts. We also use techniques of compressive sensing to propose a state estimation algorithm that can estimate the states of the system in the presence of these artifacts. Finally, we present illustrative examples of our main results using both pedagogical examples and real data collected from a wearable EEG device.

\end{abstract}

%%%%%%%%%%%%%%%%%%%%%%%%%%%%%%%%%%%%%%%%%%%%%%%%%%%%%%%%%%%%%%%%%%%%%%%%%%%%%%%%
\section{Introduction}

Fractional-order dynamical systems (FODS) have been successfully used to accurately model dynamics which undergo nonexponential power-law decay, and have long-term memory or fractal properties \cite{moon2008chaotic,lundstrom2008fractional,werner2010fractals,turcott1996fractal,thurner2003scaling,teich1997fractal}. In particular, fractional-order models have been used in domains such as gas dynamics \cite{chen2010anomalous}, viscoelasticity \cite{jaishankar2012power}, chaotic systems \cite{petravs2011fractional}, and biological swarming \cite{west2014networks}, just to mention a few. With the advent of cyber-physical systems (CPS), the usefulness of FODS becomes even more apparent, since we have to model the relationship between the spatial and temporal evolution of complex networks \cite{xuecps}, \cite{xue2017reliable}.
\par Inspired by the recent spate of application of FODS to model the spatiotemporal properties of complex physiological signals such as electroencephalogram (EEG), electrocardiogram (ECG), electromyogram (EMG), and blood-oxygen-level dependent (BOLD) imaging \cite{magin2006fractional}, \cite{baleanu2011fractional} in the context of neurophysiological applications, the estimation of the states of a fractional-order dynamical system plays a key role in assessing how the corresponding states are evolving, which unveil information about brain function. We are particularly interested in EEG signals, that are known to be prone to disturbances which are not cerebral in origin, which are known as \emph{artifacts} in the neuroscience literature \cite{britton2016electroencephalography}. There are innumerable sources for these artifacts, some of the common ones being artifacts due to the blinking of the eye (called blink artifact), the potential difference between the EEG electrode and its lead (called electrode pop artifact), and muscular motions of chewing or swallowing (called glossokinetic artifact). Therefore, it becomes imperative to develop a set of tools which can robustly estimate the dynamics of the EEG signals even in the presence of these artifacts. Note that although the term artifacts is almost always used to describe behavior which is not of a cerebral nature in EEG signals, the notion of artifacts can be readily extended to characterize the disturbances that are not consistent with the stationary evolution of the system's dynamics - in particular FODS.
\par In this paper, we develop a systematic framework to estimate the states of a fractional-order dynamical system when artifacts are associated with the measurement process. To do this, we are inspired by recent research on resilient state estimation of dynamical systems in the presence of attacks. For instance, \cite{shoukry2017secure,pajic2014robustness,pajic2017attack,hu2018state,mishra2015secure,shi2017finite,an2018secure,mo2015secure,forti2017secure,pasqualetti2013attack,fawzi2014secure,ding2018survey,fawzi2011,mo2014resilient} deal with the problem of identification and state estimation of cyber-physical systems in the presence of these attacks. Notice that artifacts are mainly `attacks by nature', which we try to cope in the context of FODS, and are not assumed to follow any stochastic process or dynamics. Hence, inspired by the techniques outlined in the former, we develop a method to  estimate the states of a fractional-order dynamical system in the presence of artifacts, which, to the best of our knowledge, has not been previously analyzed in the context of FODS.
\par A few works already exist in the domain of state estimation of fractional-order systems \cite{sabatier2012observability,sierociuk2006fractional,Safari1,Safari2,miljkovic2017ecg}. However, most of these assume the noise associated with the measurement process to have a certain stochastic characterization. The work that comes closest to ours is \cite{miljkovic2017ecg}, but even that outlines a method for artifact correction using fractional calculus and a median filter.
\par The main theoretical contributions of our paper are as follows. First, we provide necessary and sufficient conditions that ensure the estimation of the states of a discrete-time fractional-order dynamical system in the presence of arbitrary artifacts. Next, we propose an algorithm that can estimate the states in the presence of artifacts in the sensor measurements. Finally, we borrow techniques of compressive sensing to improve the algorithm such that it becomes computationally feasible.
\par The remainder of the paper is organized as follows. Section~\ref{sec:probstat} introduces the elements of the state estimation problem. Section~\ref{sec:mainres} presents necessary and sufficient conditions that guarantee the estimation of states of FODS in the presence of artifacts. Sections~\ref{sec:estim} and~\ref{sec:better_estim} consider the problems of constructing the state estimator using the techniques of compressive sensing. Finally, in Section~\ref{sec:examples}, we present illustrative examples showing the performance of our algorithm on a pedagogical example, as well as real data collected from a wearable EEG device.

\section{Problem Statement}\label{sec:probstat}

\subsection{Notation}
Throughout this paper, we use the following notations. Given a set $S$, we denote its cardinality by $|S|$, and its complement by $S^c = U \setminus S$ (the universal set $U$ will be clear from the context of the discussion). If $x \in \mathbb{R}^n$ is a vector, the \textit{support} of $x$, denoted by $\mathsf{supp}(x)$, is the set of nonzero components of $x$
\[ \mathsf{supp}(x) = \{ i \in \{1, \ldots, n\} \: | \: x_i \neq 0 \}. \]
The number of nonzero elements in $x$ will be denoted by $\|x\|_{\ell_0}$, that is
\[ \|x\|_{\ell_0} = |\mathsf{supp}(x)|. \]
Further, if $K \subset \{1,\ldots,n\}$, we define $\mathcal{P}_{K}$ to be the projection map onto the components of $K$.
Also, given a matrix $M \in \mathbb{R}^{m \times n}$, the $i$-th row of $M$, $i \in \{1,\ldots,m\}$ is denoted by $M_i \in \mathbb{R}^n$. The \textit{row support} of $M$ is defined to be the set of nonzero rows of $M$
\[ \mathsf{rowsupp}(M) = \{ i \in \{1,\ldots,m\} \: | \: M_i \neq 0_n \}, \]
where $0_n$ is the vector of zeros with dimension $n$. For matrices, like vectors, the number of nonzero rows of $M$ will be denoted by $\|M\|_{\ell_0}$, and given by
\[ \|M\|_{\ell_0} = |\mathsf{rowsupp}(M)|. \]

\subsection{Problem Statement}
\subsubsection{System Model}
Consider a linear discrete-time fractional-order dynamical model described as follows
\begin{align}\label{eq:frac_model}
    \Delta^\alpha x[k+1] &= Ax[k] \nonumber \\
    y[k] &= Cx[k] + e[k],
\end{align}
where $x \in \mathbb{R}^n$ is the state for time step $k \in \mathbb{N}$ and $y \in \mathbb{R}^p$ is the output vector. $A \in \mathbb{R}^{n \times n}$ is the system matrix and $C \in \mathbb{R}^{p \times n}$ is the sensor measurement matrix. The vector $e \in \mathbb{R}^p$ represents the model of errors, i.e., the \emph{artifacts}. Note that for the special case of EEG signals, these could be some of the EEG artifacts described in the Introduction. If there are no artifacts in channel $i \in \{1,\ldots,p\}$, then $e_i[k] = 0$ and the output $y_i[k]$ of channel $i$ is not corrupted. As a consequence, the sparsity pattern of $e[k]$ gives the set of sensors where there are artifacts present. Note that the system model is similar to a classic discrete-time linear time-invariant model except for the inclusion of the fractional derivative, whose expansion and discretization for the $i$-th state, $1 \leq i \leq n$, can be written as
\begin{equation}\label{eq:frac_deriv}
    \Delta^{\alpha_i} x_i[k] = \sum_{j=0}^k \psi (\alpha_i,j) x_i[k-j],
\end{equation}
where $\alpha_i$ is the fractional order corresponding to state $i$ and
\[ \psi(\alpha_i,j) = \frac{\Gamma(j-\alpha_i)}{\Gamma(-\alpha_i) \Gamma(j+1)}, \]
with $\Gamma(\cdot)$ being the gamma function defined by $\Gamma(z) = \int_0^{\infty} s^{z-1} e^{-s} \: \mathrm{d}s$ for all complex numbers $z$ with $\Re (z) > 0$ \cite{DzielinskiFOS}.

\subsubsection{Artifact Model}

We assume in this paper that the measurement channels in which the artifacts are present do not change over time, but can be arbitrary. Note that this is a realistic assumption for our problem since neurophysiological studies of EEG artifacts indicate that they generally occur in specific channels without switching channels arbitrarily and have highly conserved morphological structure \cite{Debener2010}. Therefore, we will talk about the \emph{number} of channels which are affected by artifacts, which will allow us to make statements about unambiguously recovering the initial state $x[0]$ even in the presence of a given number of artifacts.

\subsubsection{Central Estimator Model}
We will also assume the presence of a central estimator, whose job is to receive the output $y_i[k]$, $i \in \{1,\ldots,p\}$ of each sensor at every time step $k$ and, from this information, estimate the initial state $x[0]$. We also assume that the estimator has knowledge of the system matrix $A$, the fractional-order coefficients $\{\alpha_i\}_{i=1}^n$, and the sensor measurement matrix $C$, which can be retrieved using the methods outlined in \cite{gupta2018}. From~\eqref{eq:frac_model}, we note that the problems of estimating $x[k]$ and $x[0]$ are exactly equivalent. Therefore, our focus will be on estimating $x[0]$.
\par Based on the above ingredients, the problem we consider in this paper is as follows.
\begin{probstat}
Given the system matrix $A \in \mathbb{R}^{n \times n}$, the sensor measurement matrix $C \in \mathbb{R}^{p \times n}$, the fractional-order coefficients $\{\alpha_i\}_{i=1}^n$ corresponding to each state and the outputs $y_i[k] \in \mathbb{R}^p$ of sensor $i$, $i \in \{1,\ldots,p\}$, then for every time step $k \in \mathbb{N}$ in~\eqref{eq:frac_model}, we aim to estimate the initial state $x[0] \in \mathbb{R}^n$.
\end{probstat}
In what follows, we first provide necessary and sufficient conditions that guarantee the feasibility of the problem, followed by an efficient algorithm to estimate the states.

\section{Main Results}\label{sec:mainres}

Prior to going into our main results, we review some essential theory for fractional-order systems, including \mbox{closed-form} expressions for the state dynamics. Using the expansion of the fractional-order derivative in~\eqref{eq:frac_deriv}, the evolution of the state vector can be written as follows
\begin{equation}\label{eq:state_evol}
    x[k+1] = Ax[k] - \sum_{j=1}^{k+1} D(\alpha,j) x[k+1-j],
\end{equation}
where $D(\alpha,j) = \text{diag}(\psi(\alpha_1,j),\psi(\alpha_2,j),\ldots,\psi(\alpha_n,j))$. Alternatively,~\eqref{eq:state_evol} can be written as
\begin{equation}
    x[k+1] = \sum_{j=0}^k A_j x[k-j],
\end{equation}
where $A_0 = A - D(\alpha,1)$ and $A_j = -D(\alpha,j+1)$ for $j \geq 1$. Defining matrices $G_k$ as
\begin{equation}
    G_k =
    \begin{cases*}
      I_n & $k = 0$, \\
      \displaystyle \sum_{j=0}^{k-1} A_j G_{k-1-j} & $k \geq 1$,
    \end{cases*} 
\end{equation}
we can state the following result.
\begin{lemma}[\cite{guermah2008}]\label{lemma:state_prop}
The solution to the system described by~\eqref{eq:frac_model} is given by
\begin{equation}
    x[k] = G_k x[0].
\end{equation}
\end{lemma}
Let $x[0] \in \mathbb{R}^n$ be the initial state of the plant and let $y[0],\ldots,y[k-1] \in \mathbb{R}^p$ be the sensor measurement outputs that are available to the central estimator in $k$ time steps. Using~\eqref{eq:frac_model} and Lemma~\ref{lemma:state_prop}, we can write
\begin{equation}
    y[k] = CG_k x[0] + e[k],
\end{equation}
where $e[k]$ represents the artifacts across different measurement channels. If $K \subset \{ 1,\ldots,p \}$ denotes the set of sensors which are disturbed by artifacts, we can use our assumption that the channels in which the artifacts are present do not change over time to conclude that $\mathsf{supp}(e[k]) \subset K$. After receiving the $k$ sensor outputs $y[0],\ldots,y[k-1]$, the central estimator $\mathcal{C}:(\mathbb{R}^p)^k \to \mathbb{R}^n$ estimates the initial state $x[0]$ of the plant. The estimation is \textit{correct} if $\mathcal{C}(y[0],\ldots,y[k-1]) = x[0]$. More formally, we introduce the following definition.
\begin{definition}
We say that the central estimator $\mathcal{C}:(\mathbb{R}^p)^k \to \mathbb{R}^n$ can correctly estimate the initial state $x[0]$ in the presence of artifacts in $q$ channels in $k$ time steps if, for any $x[0] \in \mathbb{R}^n$ and \emph{any} sequence of vectors $e[0],\ldots,e[k-1] \in \mathbb{R}^p$ such that $\mathsf{supp}(e[k]) \subset K$ with $|K| = q$, we have $\mathcal{C}(y[0],\ldots,y[k-1]) = x[0]$ with $y[k] = CG_k x[0] + e[k]$.
\end{definition}
Subsequently, the first question one can ask is as follows: What are the required conditions that ensure that the estimation is correct? We provide an answer to this in the next couple of results. First, we consider necessary conditions.
\begin{theorem}
Let $k \in \mathbb{N} \setminus \{0\}$ and $E_{q,k}$ be the set of error vectors $(e[0],\ldots,e[k-1]) \in (\mathbb{R}^p)^k$ that satisfy for all $k' \in \{0,\ldots,k-1\}$, $\mathsf{supp}(e[k']) \subset K$ for some $K \in \{1,\ldots,p\}$ with $|K| = q$. Then, the following statements are equivalent:
\begin{enumerate}
    \item[(1)] There \textit{does not exist} a central estimator that can recover the initial state $x[0]$ in $k$ time steps with artifacts across $q$ channels.
    \item[(2)] There exists $x_a, x_b \in \mathbb{R}^n$, $x_a \neq x_b$ and $(e_a[0],\ldots,e_a[k-1]), (e_b[0],\ldots,e_b[k-1]) \in E_{q,k}$ such that $G_k x_a + e_a[k'] = G_k x_b + e_b[k']$ for all $k' \in \{0,\ldots,k-1\}$.
\end{enumerate}
\end{theorem}
\begin{proof}
The proof of this Theorem readily follows from the distinguishability arguments in the context of observable states, where, given two distinct values $x_a, x_b$ with $x_a \neq x_b$ explaining the data collected with less than $q$ affected channels, it is impossible to unambiguously recover $x[0]$.
\end{proof}
\par We then have a simple necessary and sufficient condition in terms of the matrices $C$ and $G_k$ to characterize the effect of correctly estimating the system states with artifacts being present in $q$ channels.
\begin{theorem}
Let $k \in \mathbb{N} \setminus \{0\}$. Then, the following statements are equivalent:
\begin{enumerate}
    \item[(1)] The central estimator $\mathcal{C}$ can recover $x[0]$ in $k$ time steps in the presence of artifacts across $q$ channels.
    \item[(2)] For all $z \in \mathbb{R}^n \setminus \{0\}$, $| \mathsf{supp}(CG_0 z) \cup \mathsf{supp}(CG_1 z) \cup \ldots \cup \mathsf{supp}(CG_{k-1}z) | > 2q$.
\end{enumerate}
\end{theorem}

\subsection*{Comparisons with LTI systems}

Similar results have been obtained by Fawzi et al. \cite{fawzi2011} for LTI systems. However, there are some crucial differences. Fawzi et al. consider the problem of dealing with errors injected by a malicious agent in the case of a discrete-time linear time-invariant system without inputs. Their version of the above theorem is a generalized criterion for the observability of a linear dynamical system with attacks, which implies that the initial state $x[0]$ can only be recovered in $k$ time steps if the observability matrix given by $\mathcal{O}_{\mathrm{LTI}} = \begin{bmatrix} C^{\mathsf{T}} & A^{\mathsf{T}} C^{\mathsf{T}} & \ldots & (A^{\mathsf{T}})^{k-1} C^{\mathsf{T}} \end{bmatrix}^{\mathsf{T}}$ has full rank. Furthermore, the maximum number of correctable errors cannot increase beyond $k = n$ measurements, which follows immediately from the Cayley-Hamilton theorem, since, for any $z$, $k \geq n$, we have $\mathsf{supp}(CA^k z) \subset \mathsf{supp}(Cz) \cup \ldots \cup \mathsf{supp}(CA^{k-1}z)$.
\par In general, for a linear discrete-time fractional-order system modeled by~\eqref{eq:frac_model} and the matrices $G_k$, the necessary and sufficient condition for observability is the existence of a finite time $k'$ (which may be greater than $n$) such that $\text{rank}(\Xi_{k'}) = n$, where the observability matrix $\Xi_{k'} = \begin{bmatrix} CG_0 & CG_1 & \ldots & CG_{k'-1} \end{bmatrix}^{\mathsf{T}}$ \cite{guermah2008}. Therefore, a remarkable fact is that for a discrete-time fractional-order system, the maximum number of correctable artifacts \textit{can increase} beyond $n$ measurements. This is due to the fact that the terms in the observability matrix for fractional-order systems do not constitute a power series; hence we cannot apply the Cayley-Hamilton theorem as in the linear time-invariant case. Specifically, the terms $G_{k'}$ in the observability matrix $\Xi_{k'}$ are composed of the terms $A_j$, $j = 0,1,\ldots,k-1$, which increase with the time step. Put simply, fractional-order dynamics aggregate the effects of \textit{all time} for each new iteration of the state, and hence, this long-term memory becomes relevant in determining the minimum number of time steps in which the system becomes observable. Also, the property of being able to increase the number of correctable artifacts beyond $n$ measurements is ideal from the perspective of our problem, since it assures us that taking more measurements will not go in vain.

Next, we present a result on the number of correctable artifacts as a function of the number of measurements $\tau$.
\begin{theorem}
Let $\tau \in \mathbb{N} \setminus \{0\}$ be such that $p\tau \geq k'$, $p \in \mathbb{N} \setminus \{0\}$ where $k' \geq n$ is the minimum index such that $\text{rank}(\Xi_{k'})=n$. If the initial state $x[0]$ can be estimated in $\tau$ time steps in the presence of artifacts in $q$ channels, then
\[ q < \frac{p-\lfloor (k'-1)/\tau \rfloor}{2} \leq \frac{p-k'/\tau+1}{2}.\]
\end{theorem}
\begin{remark}
The problem of actually computing the number of channels affected by artifacts is nontrivial and involves checking that the nullspace of
\[ \begin{bmatrix}
\mathcal{P}_{K^c} CG_0 \\ \mathcal{P}_{K^c} CG_1 \\ \vdots \\ \mathcal{P}_{K^c} CG_{\tau-1}
\end{bmatrix}\]
is not trivial, which, in the worst case, requires computing the rank of $2^p$ matrices, when $K \subset \{1,\ldots,p\}$.
\end{remark}

\section{Constructing the Central Estimator}\label{sec:estim}

In this section, we will focus on actually constructing the central estimator that can estimate the initial state $x[0]$. Consider the central estimator $\mathcal{C}_0^k: (\mathbb{R}^p)^k \to \mathbb{R}^n$ defined in a way that $C_0^k(y[0],\ldots,y[k-1])$ is the optimal $x$ solution for the optimization problem
\begin{equation}\label{eq:opt_prob}
\begin{aligned}
& \underset{x \in \mathbb{R}^n, K \subset \{1,\ldots,p\}}{\text{minimize}}
& & |K| \\
& \text{subject to}
& & \mathsf{supp}(y[k'] - CG_{k'}x) \subset K, \\
&&& \text{for} \: k' \in \{0,\ldots,k-1\}.
\end{aligned}
\end{equation}
We note that the optimization problem as stated above can have multiple solutions, since the central estimator looks for the smallest set $K$ of channels disturbed by artifacts. In such a case, we consider $\mathcal{C}_0^k(y[0],\ldots,y[k-1])$ to be \emph{any} such solution.
\par In the next result, we show that the the central estimator $\mathcal{C}_0^k$ is, in some sense, \emph{optimal}.
\begin{theorem}\label{th:opt_est}
Assume that $x[0]$ can be estimated in $k$ time steps in the presence of $q$ artifacts, or, $| \mathsf{supp}(CG_0 z) \cup \ldots \cup \mathsf{supp}(CG_{k-1}z) |>2q$ for all $z \in \mathbb{R}^n \setminus \{0\}$. Then, $C_0^k$ correctly estimates $x[0]$, that is, for any $x[0] \in \mathbb{R}^n$ and $e[0],\ldots,e[k-1] \in \mathbb{R}^p$ such that $\mathsf{supp}(e[k]) \subset K$ with $|K| \leq q$, we have $\mathcal{C}_0^k(y[0],\ldots,y[k-1]) = x[0]$, where $y[k] = CG_k x[0] + e[k]$.
\end{theorem}
\par The implication of Theorem~\ref{th:opt_est} is that if \textit{any} estimator can estimate $x[0]$ in $k$ time steps in the presence of $q$ artifacts, then $\mathcal{C}_0^k$ also can. However, the optimization problem~\eqref{eq:opt_prob} is NP-hard in general \cite{guru2008}.

\section{The $\ell_1$ Central Estimator}\label{sec:better_estim}

Given the time step $k \in \mathbb{N} \setminus \{ 0 \}$, consider the linear mapping $\Phi_k$ defined by
\begin{align*}
    \Phi_k : &\mathbb{R}^n \to \mathbb{R}^{p \times k} \\
    &x \mapsto \big[\begin{array}{c|c|c|c}
CG_0 x & CG_1 x & \ldots & CG_{k-1} x
\end{array}\big].
\end{align*}
Further, define $Y_k$ as the $p \times k$ matrix obtained by stacking $y[0],\ldots,y[k-1]$ as columns, that is,
\[
    Y_k = \big[\begin{array}{c|c|c|c}
y[0] & y[1] & \ldots & y[k-1]
\end{array}\big] \in \mathbb{R}^{p \times k}.\]
We have already seen that the central estimator $\mathcal{C}_0^k$ computes $x[0]$ from the measurements $y[0],\ldots,y[k-1]$ as the solution of the ``$\ell_0$-norm'' optimization problem
\begin{equation}
    \mathcal{C}_0^k(y[0],\ldots,y[k-1]) = \argmin_{x \in \mathbb{R}^n} \| Y_k - \Phi_k x \|_{\ell_0}.
\end{equation}
\begin{remark}\label{remark:l0-l1}
In \cite{CandesTao}, Candes and Tao show that under certain mild conditions, the ``$\ell_0$-norm'' can be replaced by the $\ell_1$-norm (that is, an $\ell_1$-norm relaxation), which leads to a linear program, for which the solutions are more computationally tractable.
\end{remark}
From the argument in Remark~\ref{remark:l0-l1}, we have that a central estimator where the ``$\ell_0$-norm'' is replaced by the $\ell_1$-norm can perform the same job without any change in the optimal value of the optimization problem. More generally, following a similar approach to that proposed in \cite{fawzi2011}, given $r \geq 1$, if we measure the magnitude of a row of a matrix by its $\ell_r$-norm in $\mathbb{R}^k$, then, the so-called ``$\ell_1 / \ell_r$ central estimator'' can be represented as follows
\begin{equation}\label{eq:convex_opt}
    C_{1,r}^k(y[0],\ldots,y[k-1]) = \argmin_{x \in \mathbb{R}^n} \| Y_k - \Phi_k x \|_{\ell_1 / \ell _r},
\end{equation}
where,
\[ \|M\|_{\ell_1 / \ell_r}  = \sum_{i=1}^p \|M_i\|_{\ell_r}.\]
Note that~\eqref{eq:convex_opt} is a convex optimization problem and can be tractably solved.
\subsection{State estimation capability of the $\ell_1 / \ell_r$ central estimator}
Next, we present a result that quantifies the state estimation capability of the $\ell_1 / \ell_r$ central estimator.
\begin{theorem}\label{correction_prop}
The following statements are equivalent:
\begin{enumerate}
    \item[(1)] The $\ell_1 / \ell_r$ central estimator $\mathcal{C}_{1,r}^k$ can correctly estimate $x[0]$ in $k$ time steps in the presence of $q$ artifacts.
    \item[(2)] For all $K \subset \{1,\ldots,p\}$, with $|K| = q$ and for all $\Sha = \Phi_k z$ with $z \in \mathbb{R}^n \setminus \{0\}$, we have
    \[ \sum_{i \in K} \| \Sha_i \|_{\ell_r} < \sum_{i \in K^c} \| \Sha_i \|_{\ell r}. \]
\end{enumerate}
\end{theorem}

\begin{remark}
Those who are familiar with the techniques of compressed sensing will immediately notice that the condition $\sum_{i \in K} \| \Sha_i \|_{\ell_r} < \sum_{i \in K^c} \| \Sha_i \|_{\ell_r}$ is simply a restatement of the ``nullspace property'' that gives necessary and sufficient conditions for the recovery of sparse signals using $\ell_1$-norm relaxation \cite{davenport2011}.
\end{remark}

\subsection{An equivalent sufficient condition}
As evident from the previous section, the $\ell_1 / \ell_r$ central estimator $\mathcal{C}_{1,r}^k$ can correctly estimate $x[0]$ in $k$ time steps in the presence of $q$ artifacts, if and only if $\sum_{i \in K} \| \Sha_i \|_{\ell_r} < \sum_{i \in K^c} \| \Sha_i \|_{\ell_r}$ for all $K \subset \{1,\ldots,p\}$ with $|K|=q$ and for all $\Sha = \Phi_k z$ with $z \in \mathbb{R}^n \setminus \{ 0 \}$. However, it is not easy to check the above inequality for every $z \in \mathbb{R}^n \setminus \{0\}$.
\par In this section, we will propose sufficient conditions that are easier to check and will let us conclude whether the central estimator $\mathcal{C}_{1,r}^k$ can estimate the system states in $k$ time steps even in the presence of $q$ artifacts. From Theorem~\ref{correction_prop}, for a given $K \subset \{1,\ldots,p\}$ and a given $z \in \mathbb{R}^n \setminus \{0\}$, we have
\begin{align*}
    &\sum_{i \in K} \| (\Phi_k z)_i \|_{\ell_r} < \sum_{i \in K^c} \| (\Phi_k z)_i \|_{\ell_r} \\
    \iff &\frac{\| (\Phi_k z)_K \|_{\ell_1 / \ell_r}}{\| (\Phi_k z)_{K^c} \|_{\ell_1 / \ell_r}} < 1,
\end{align*}
where $(\Phi_k z)_K \in \mathbb{R}^{|K| \times k}$ denotes the $|K| \times k$ matrix obtained by retaining only the rows in $K$ from $\Phi_k z$ (similarly for $( \Phi_k z )_{K^c}$). Using this notation, we can say that the $\ell_1 / \ell_r$ central estimator can recover the correct $x[0]$ in the presence of $q$ artifacts if and only if
\[ \sup_{\substack{K \subset \{1,\ldots,p\} \\ |K|=q}} \sup_{z \in \mathbb{R}^n \setminus \{0\}} \frac{\| (\Phi_k z)_K \|_{\ell_1 / \ell_r}}{\| (\Phi_k z)_{K^c} \|_{\ell_1 / \ell_r}} < 1.\]
We now propose a way which enables the efficient assessment of the above inequality. Specifically, we know that when the chosen norm is $\ell_r = \ell_2$ and $L$ is a linear operator, we can use the minimum and maximum singular values of $L$, $\sigma_{\text{min}}$ and $\sigma_{\text{max}}$ to write $\sigma_{\text{min}} \|z\|_2 \leq \| Lz \|_2 \leq \sigma_{\text{max}} \| z \|_2$.
\par We have, from the definition of the $\ell_1 / \ell_r$ norm, $\| (\Phi_k z)_K \|_{\ell_1 / \ell_r} = \sum_{i \in K} \| (\Phi_k z)_i \|_{\ell_r}$. Denote by $\Phi_{k,i}$ the linear map from $\mathbb{R}^n$ to $\mathbb{R}^k$ such that $\Phi_{k,i} z = (\Phi_k z)_i$ for all $z \in \mathbb{R}^n$. This map can be represented by the matrix
\[ \begin{bmatrix}
\mathcal{P}_{\{i\}} CG_0 \\ \mathcal{P}_{\{i\}} CG_1 \\ \vdots \\ \mathcal{P}_{\{i\}} CG_{k-1}
\end{bmatrix},\]
where $\mathcal{P}_{\{i\}}$ is the projection map onto component $i$. Now, assume $K \subset \{1,\ldots,q\}$ is such that $|K|=q$ is fixed. The numerator of the fraction $\frac{\| (\Phi_k z)_K \|_{\ell_1 / \ell_r}}{\| (\Phi_k z)_{K^c} \|_{\ell_1 / \ell_r}}$ can be written as
\[ \| (\Phi_k z)_K \|_{\ell_1 / \ell_r} = \sum_{i \in K} \| \Phi_{k,i} z \|_{\ell_r} \leq \sum_{i \in K} \| \Phi_{k,i} \|_{\ell_r} \| z \|_{\ell_r}.\]
If $\beta = \max_{i=1,\ldots,p}\| \Phi_{k,i} \|_{\ell_r}$, then
\begin{equation}\label{eq:num_bound}
  \|(\Phi_k z)_K \|_{\ell_1 / \ell_r} \leq q\beta \| z \|_{\ell_r},  
\end{equation}
since $|K|=q$. Next, consider the denominator of $\frac{\| (\Phi_k z)_K \|_{\ell_1 / \ell_r}}{\| (\Phi_k z)_{K^c} \|_{\ell_1 / \ell_r}}$. Assume that $\ell_r = \ell_2$. This allows us to write $\| \Phi_{k,i} z \|_{\ell_2} \geq \sigma_{\text{min}} (\Phi_{k,i}) \|z\|_{\ell_2}$, where $\sigma_{\text{min}} (\Phi_{k,i}) \|z\|_{\ell_2}$ is the smallest singular value of the linear map $\Phi_{k,i}$. Next, denoting $\alpha = \min_{i=1,\ldots,p} \sigma_{\text{min}} (\Phi_{k,i})$, we get
\begin{equation}\label{eq:denom_bound}
    \| (\Phi_k z)_{K^c} \|_{\ell_1 / \ell_2} = \sum_{i \in K^c} \| \Phi_{k,i} z \|_{\ell_2} \geq (p-q) \alpha \|z\|_2.
\end{equation}
Using~\eqref{eq:num_bound} and~\eqref{eq:denom_bound}, we get
\begin{equation}
    \sup_{z \in \mathbb{R}^n \setminus \{0\}} \frac{\| (\Phi_k z)_K \|_{\ell_1 / \ell_2}}{\| (\Phi_k z)_{K^c} \|_{\ell_1 / \ell_2}} < \frac{q\beta}{(p-q)\alpha}.
\end{equation}
Thus, the $\ell_1 / \ell_r$ central estimator $\mathcal{C}_{1,r}^k$ can correctly recover $x[0]$ in $k$ time steps in the presence of $q$ artifacts if $\frac{q \beta}{(p-q) \alpha} < 1$. Rearranging this inequality yields $q < \frac{p \alpha}{\alpha + \beta}$. So, artifacts in at least $\lceil \frac{p \alpha}{\alpha + \beta} - 1 \rceil$ channels can be corrected in $k$ time steps. Simply speaking, we can do state estimation under setups where arbitrary artifacts affect at most $q$ channels.

\section{Illustrative Examples}\label{sec:examples}
In this section, we show the performance of the proposed estimator first on a pedagogical toy example, and then on a more realistic system consisting of EEG signals under possible artifacts.
\subsection{Pedagogical Example}
First, we consider the performance of the $\ell_1 / \ell_2$ central estimator on a system of size $n=p=4$ with
\[ A = \begin{bmatrix} 0 & 1 & 0 & 0 \\ 0.0021 & -0.0273 & -10.4940 & 0.8629 \\ 0 & 0 & 0 & 1 \\ 0.0053 & -0.0682 & -1.7351 & 2.1573 \end{bmatrix}, \]
$C = I_{4 \times 4}$, and the fractional-order coefficients $\alpha_1 = 0.10$, $\alpha_2 = 0.15$, $\alpha_3 = 0.60$, and $\alpha_4 = 0.70$. The choice of $C$ being the identity matrix if often suitable in physiological applications, where there are dedicated sensors to capture the evolution of each system state \cite{Xue2016}. We assume that there are arbitrary artifacts associated with some measurements of the output $y_1$, and that the artifacts are ten times the magnitude of the states. Figure~\ref{fig:toy_example} shows the comparison between the actual and estimated states which confirms the fact that the system states have been estimated correctly in the presence of artifacts. Note that we have also implicitly verified the statements in Theorems 2 and 3, since, with $p = 4$, $q = 1$, the number of measurements $\tau = 5$, $k'=4$, and $z = e_4$, where $e_4$ is the $4$-dimensional vector of ones, we have $|\mathsf{supp}(C G_0 z) \cup \ldots \cup \mathsf{supp}(C G_4 z)| = 4 > 2q = 2$ and $q \leq (p - k'/\tau + 1)/2$. The optimization problems were solved using \texttt{CVX} \cite{cvx},\cite{grant08}.
\begin{figure}[t]
    \centering
    \includegraphics[width=0.5\textwidth]{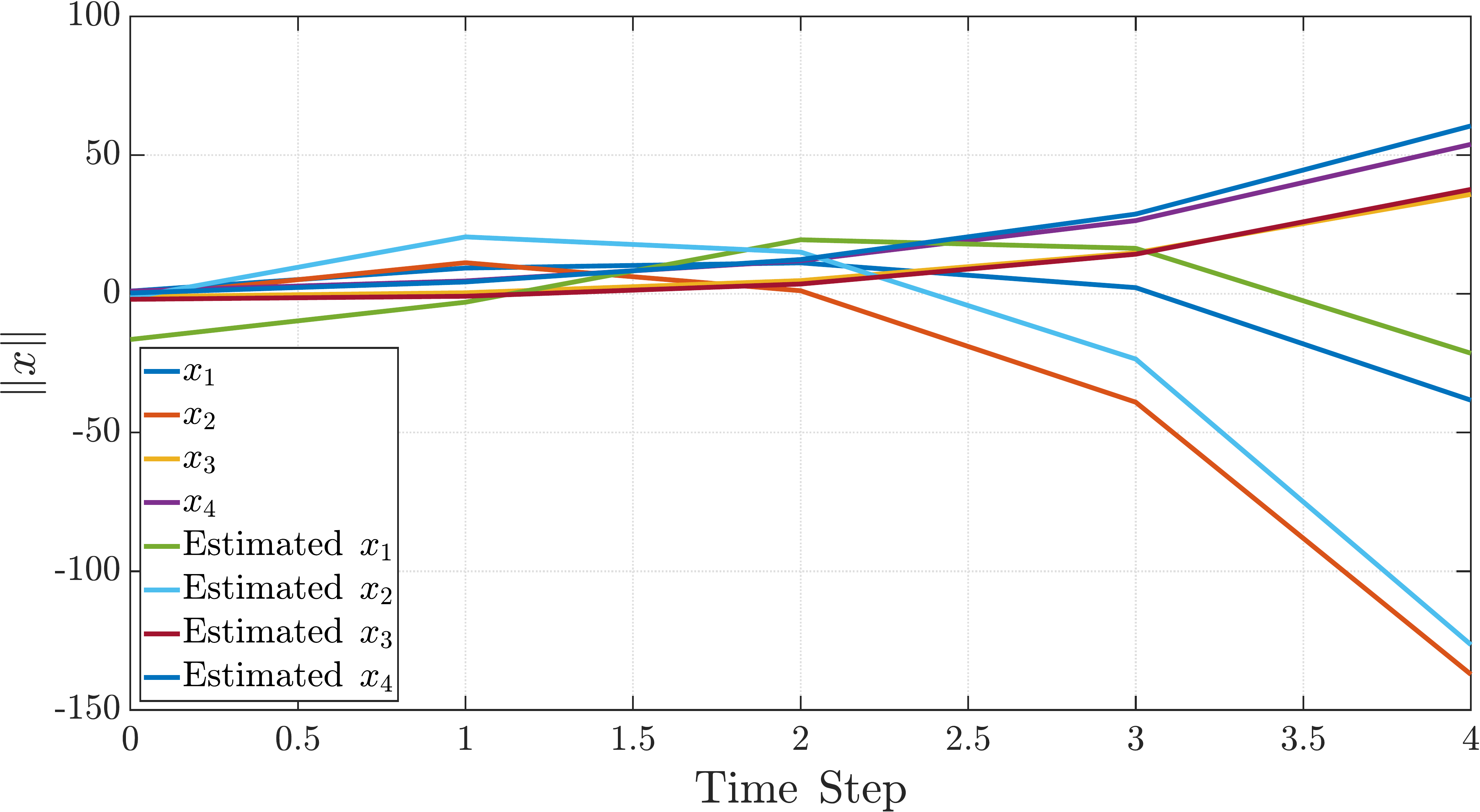}
    \caption{Actual and estimated system states for the $4 \times 4$ toy pedagogical example using the $\ell_1 / \ell_2$ central estimator.}
    \label{fig:toy_example}
\end{figure}

\begin{figure}[ht]
    \centering
    \includegraphics[width=0.4\textwidth]{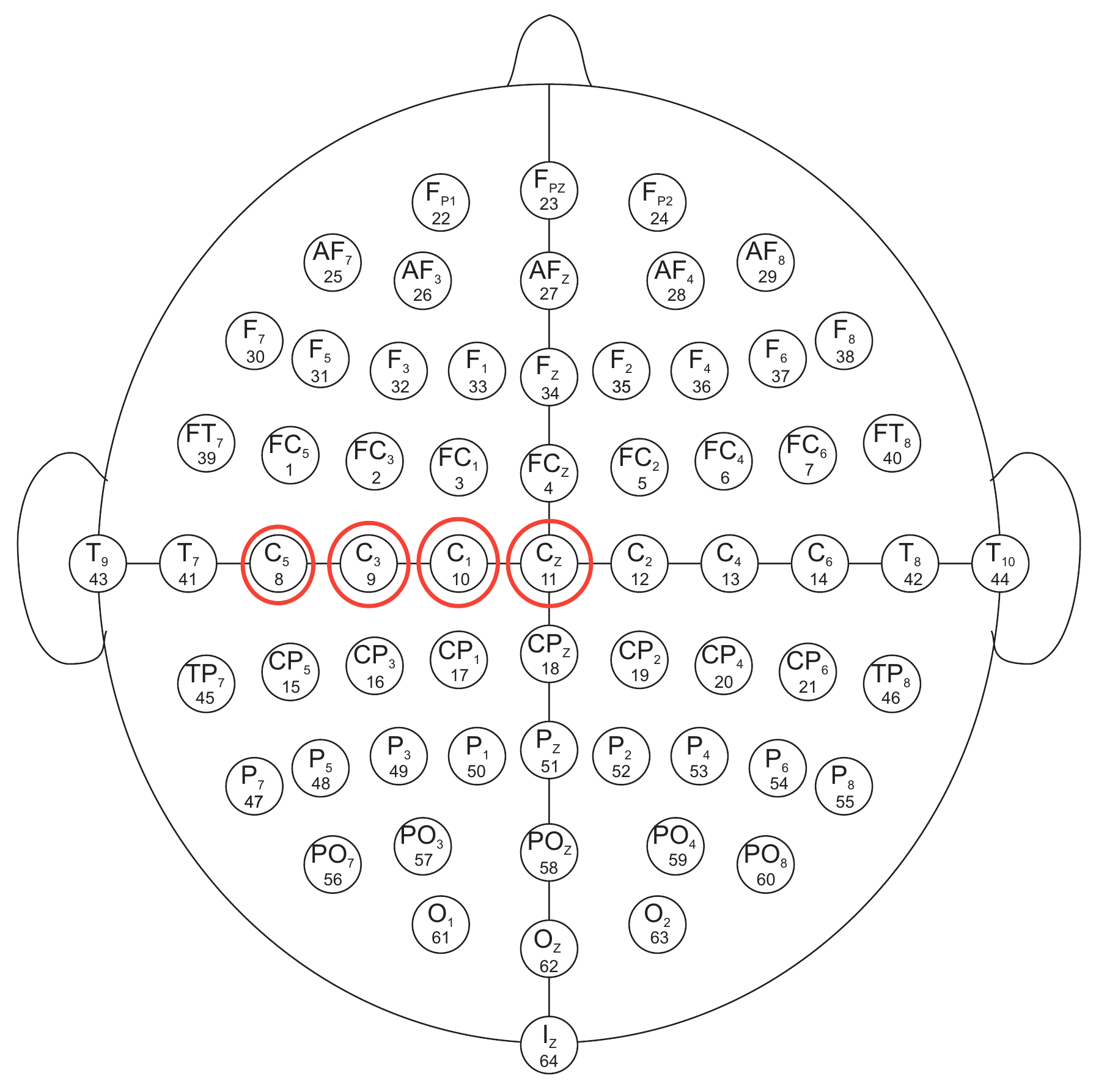}
    \caption{Sensor distribution for the measurement of EEG. The channel labels are shown along with their corresponding numbers and the selected channels are shown in red.}
    \label{fig:eeg_dist}
\end{figure}
\subsection{EEG data}
We also consider the performance of our algorithm on real neurophysiological signals. In particular, we use $150$ measurements taken from $4$ channels of a $64$-channel EEG signal which records the brain activity of subjects. The distribution of the electrodes and the selected channels, along with the corresponding labels and numbers are shown in Figure~\ref{fig:eeg_dist}. The subjects were asked to perform various motor and imagery tasks, and the channels were selected because they are over the motor cortex of the brain, and enable us to predict motor actions such as movement of the hands and feet. The data was collected using the BCI$2000$ system with a sampling rate of $160$Hz \cite{SchalkBCI}, \cite{goldberger2000physiobank}. We assume that there is an `electrode pop' artifact in the first channel that lasts for $25$ milliseconds, following which the electrode picks up noisy data only. The $150$ measurements were partitioned into windows of size $6$. This was done because the process under consideration is nonlinear, therefore, the error can become unbounded even under small perturbations, possibly due to the system identification process. The system was identified using the methods described in \cite{gupta2018}. As before, we assume that there are dedicated sensors for each system state, that is, $C = I_{4 \times 4}$.
\par Similar to the pedagogical example, we find that we can invoke the results of the paper, and, in particular, we can estimate the states with artifacts compromising one of the channels. Figure~\ref{fig:eeg_data} shows the performance of the $\ell_1 / \ell_2$ central estimator on the above data. We see that even in the presence of artifacts, we can estimate the system states fairly closely. Therefore, these simulations provide some evidence that the proposed approach might be used in the context of future neuro-wearable device applications.

\section{Conclusion}
\begin{figure}[t]
    \centering
    \includegraphics[width=0.5\textwidth]{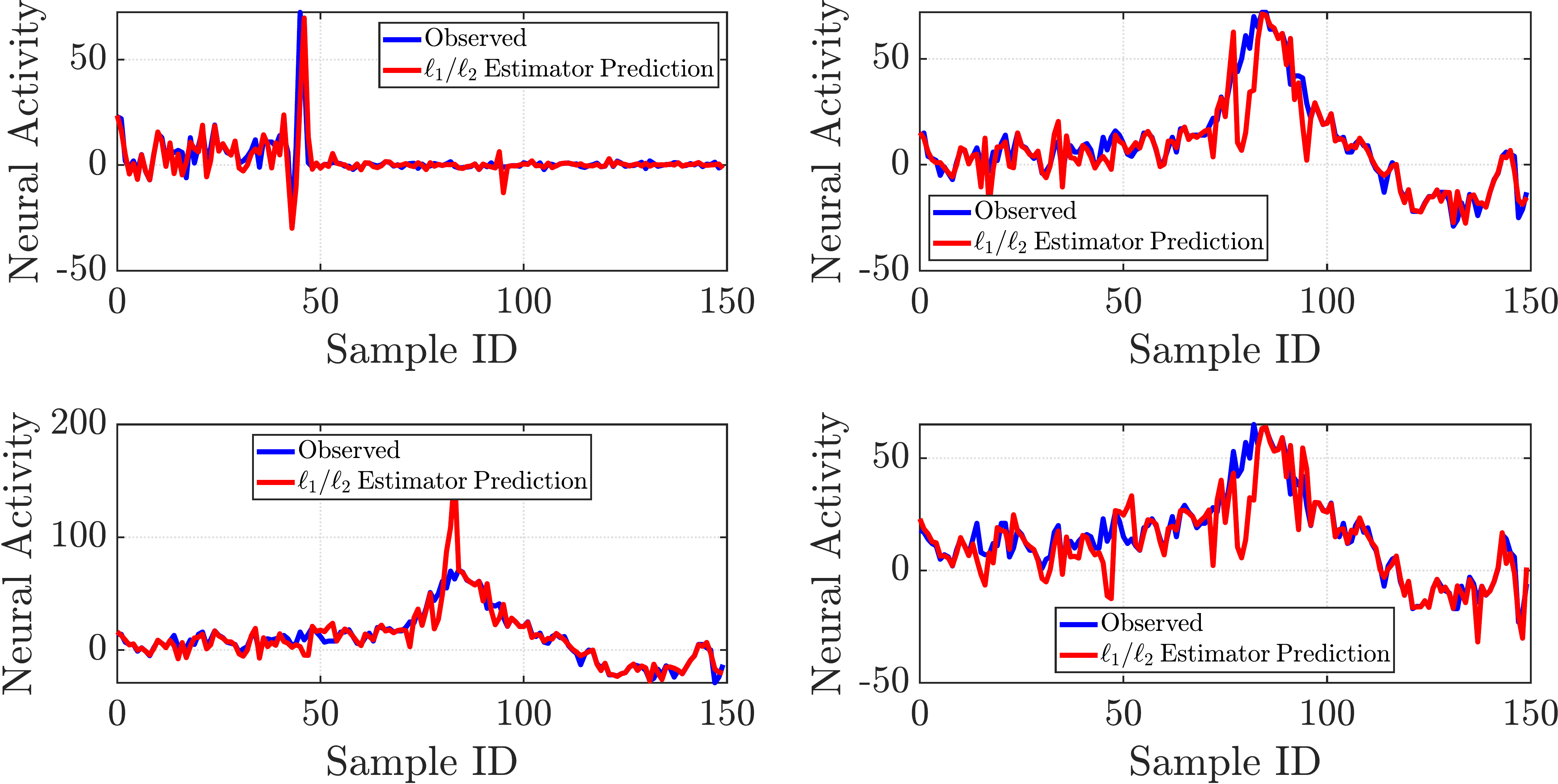}
    \caption{Performance of the $\ell_1 / \ell_2$ central estimator on data collected from a wearable EEG device.}
    \label{fig:eeg_data}
\end{figure}
In this paper, we investigated the problem of estimating the states of a discrete-time fractional-order dynamical system when there are artifacts present in some of the sensor measurements. In particular, we derive necessary and sufficient conditions that enable us to ensure that the estimation of the system states can be done even in the presence of artifacts. We further present some techniques which are inspired by tools used in compressive sensing to estimate the system states. Our results were illustrated on a toy fractional order system as well as more complex systems that represent data collected from a wearable EEG device.
\par Although FODS have found huge success in modeling the spatiotemporal properties of EEG, some of the properties accounted for by these models actually originate from unknown sources that are external to the system under consideration. Possible future work will be to model these external sources by unknown input stimuli, and then focus on state estimation of the resultant model with inputs. Also, real-time EEG activity can be monitored in order to self-regulate brain function. This is known in the literature as neurofeedback \cite{marzbani2016neurofeedback}, and it would be interesting to study how the introduction of feedback to such a system changes our perspectives on this problem.


%%%%%%%%%%%%%%%%%%%%%%%%%%%%%%%%%%%%%%%%%%%%%%%%%%%%%%%%%%%%%%%%%%%%%%%%%%%%%%%%

%%%%%%%%%%%%%%%%%%%%%%%%%%%%%%%%%%%%%%%%%%%%%%%%%%%%%%%%%%%%%%%%%%%%%%%%%%%%%%%%

%%%%%%%%%%%%%%%%%%%%%%%%%%%%%%%%%%%%%%%%%%%%%%%%%%%%%%%%%%%%%%%%%%%%%%%%%%%%%%%%
\appendix
\textbf{Proof of Theorem 2:} (1) $\implies$ (2): Assume for the sake of contradiction that there exists a $z \in \mathbb{R}^n \setminus \{0\}$ such that $| \mathsf{supp}(CG_0 z) \cup \mathsf{supp}(CG_1 z) \cup \ldots \cup \mathsf{supp}(CG_{k-1}z) | \leq 2q$. Let $e_a[k]$ and $e_b[k]$ be such that $CG_kz = e_a[k] - e_b[k]$ with $\mathsf{supp}(e_a[k]) \subset L_a, \mathsf{supp}(e_b[k]) \subset L_b$ with $|L_a|, |L_b| \leq q$ where $L_a, L_b$ are any two subsets of $\{1,\ldots,p\}$ with cardinality less than or equal to $q$ satisfying $L_a \cup L_b = \mathsf{supp}(CG_0 z) \cup \mathsf{supp}(CG_1 z) \cup \ldots \cup \mathsf{supp}(CG_{k-1}z)$. Now, suppose for $k' \in \{0,\ldots,k-1\}$, $y[k'] = CG_{k'}z + e_b[k'] = CG_{k'}\cdot 0 + e_a[k']$. This means that if the central estimator can properly recover $x[0]$ in $k$ time steps even in the presence of $q$ artifacts, then $\mathcal{C}(y[0],\ldots,y[k-1]) = z = 0$, which is impossible since $z \neq 0$.
\par (2) $\implies$ (1): The proof follows by contradiction. Suppose that the central estimator \textit{cannot} recover $x[0]$ in $k$ time steps when there are $q$ artifacts. This immediately necessitates the existence of $x_a \neq x_b$ and error vectors $e_a[0],\ldots,e_a[k-1]$ (supported on $L_a$, with $|L_a| \leq q$) and $e_b[0],\ldots,e_b[k-1]$ (supported on $L_b$, with $|L_b| \leq q$) such that $C G_{k'} x_a + e_a[k'] = C G_{k'} x_b + e_b[k']$ for all $k' \in \{0,\ldots,k-1\}$. Now, say $z = x_a - x_b \neq 0$. If we define $L = L_a \cup L_b$, then we must necessarily have $|L| \leq 2q$, and therefore, for all $k' \in \{0,\ldots,k-1\}$, $\mathsf{supp}(C G_{k'} z) \subset L$, which contradicts (2).\qed

\textbf{Proof of Theorem 3:} Our goal will be to show that there exists a $z \neq 0$, such that $|\mathsf{supp}(CG_0z) \cup \ldots \cup \mathsf{supp}(CG_{\tau-1}z)| \leq p - \lfloor (k'-1)/\tau \rfloor$. Assume $L$ to be any subset of $\{1,\ldots,p\}$ such that $|L| = \lfloor (k'-1)/\tau \rfloor$. An obvious choice would be $L = \{ 1,\ldots,\lfloor (k'-1)/\tau \rfloor \}$. Consider the linear operator $\Phi: z \in \mathbb{R}^n \mapsto (\mathcal{P}_L CG_0z,\mathcal{P}_L CG_1z,\ldots,\mathcal{P}_L CG_{\tau-1}z) \in \mathbb{R}^{|L|\tau}$, where $\mathcal{P}_L$ is a projection onto the components of $L$. Now, since the codomain of $\Phi$ is $\mathbb{R}^{|L|\tau}$, and since $|L| = \lfloor (k'-1)/\tau \rfloor < k'/\tau$, the codomain of $\Phi$ has dimension strictly less than $k'$, which means that the nullspace of $\Phi$ has \textit{at least} one vector other than the zero vector and is nontrivial. Hence, there exists a $z \neq 0$, such that $|\mathsf{supp}(CG_0z) \cup \ldots \cup \mathsf{supp}(CG_{\tau-1}z)| \subseteq L^c$, and so, by the property of cardinality of set union, $|\mathsf{supp}(CG_0z) \cup \ldots \cup \mathsf{supp}(CG_{\tau-1}z)| \leq |L^c|  = p - \lfloor (k'-1)/\tau \rfloor$.\qed

\textbf{Proof of Theorem 4:} Assume that the ordered pair $(x[0],K)$ is not optimal in~\eqref{eq:opt_prob}. This necessitates the existence of some other initial state $x_a \neq x[0]$ and error vectors $e_a[0],\ldots,e_a[k-1]$ with $\mathsf{supp}(e_a[k]) \subset K_a$ that generate the same sequence of sensor readings $y[0],\ldots,y[k-1]$. In addition, $|K_a| \leq |K| \leq q$. Therefore, we have two different initial conditions $x_a \neq x[0]$, and two different sets of error vectors $e[0],\ldots,e[k-1]$ and $e_a[0],\ldots,e_a[k-1]$ that explain the sequence of observed sensor readings while corresponding to less than $q$ correctable artifacts. This contradicts our assumption that $(x[0],K)$ is not optimal in~\eqref{eq:opt_prob}, and we are done.\qed

\textbf{Proof of Theorem 5:} (1) $\implies$ (2): We resort to contradiction and assume that (2) does not hold. Then, there exists $K \subset \{1,\ldots,p\}$ with $|K|=q$, and $\Sha = \Phi_k z \in \mathbb{R}^{p \times k}$ with $z \neq 0$ such that $\sum_{i \in K} \|\Sha_i\|_{\ell_r} \geq \sum_{i \in K^c} \|\Sha_i\|_{\ell_r}$. Suppose $x[0]=0$ and define $K$-supported error vectors $e[k']$ for $k' \in \{ 0,1,\ldots,k-1 \}$ by
\[e_i[k'] =
    \begin{cases*}
      \Sha_{i,k'} & \text{if} $i \in K$, \\
      0 & \text{otherwise}.
    \end{cases*}\]
Now, from the output equation $y[k'] = CG_{k'}x[0] + e[k'] = e[k']$, and $Y_k$ be the matrix obtained by stacking the sensor readings $y[0],\ldots,y[k-1]$ in columns. Of course, since $K$ denotes the set of channels in which artifacts are present, $\mathsf{rowsupp}(Y_k) = K$ and $(Y_k)_i = (\Phi_k z)_i$ for all $i \in K$. Subsequently, we have
\begin{align*}
    \|Y_k - \Phi_k z\|_{\ell_1 / \ell_r} &= \sum_{i=1}^n \| (Y_k - \Phi_k z)_i \|_{\ell_r}\\
    &= \sum_{i \in K^c} \| \Sha_i \|_{\ell_r} \leq \sum_{i \in K} \| \Sha_i \|_{\ell_r}\\
    &= \sum_{i=1}^n \| (Y_k - \Phi_k x[0])_i \|_{\ell_r}\\
    &= \| Y_k - \Phi_k x[0] \|_{\ell_1 / \ell_r}.
\end{align*}
This means that there exists a $z \neq 0$ such that the value of the objective function in~\eqref{eq:convex_opt} is smaller at $z$ than at $x[0] = 0$. Therefore, $C_{1,r}^k$ fails to reconstruct the initial state from the sensor outputs, which implies that (2) must be true.
\par (2) $\implies$ (1): We suppose that (1) is not true, that is, there exists $x[0]$ with error vectors $e[0],\ldots,e[k-1]$ with $\mathsf{supp}(e[k']) \subset K$ with $|K| = q$ such that $\mathcal{C}_{1,r}^k(y[0],\ldots,y[k-1]) \neq x[0]$ with $y[k'] = CG_{k'}x[0] + e[k']$ for all $k' \in \{0,\ldots,k-1\}$. Since the central estimator fails to reconstruct $x[0]$ from the sensor outputs, this means that in the optimization problem~\eqref{eq:convex_opt}, there exists an $\hat{x} \neq x[0]$ that achieves a lower objective than $x[0]$. Next, define $z = \hat{x} - x[0] \neq 0$, $\Sha = \Phi_k z = \mathcal{U} - \mathcal{V}$ with $\mathcal{U} = Y_k - \Phi_k x[0]$ and $\mathcal{V} = Y_k - \Phi_k \hat{x}$. Then,
\[ \sum_{i \in K} \| \Sha_i \|_{\ell_r} = \sum_{i \in K} \| \mathcal{U}_i - \mathcal{V}_i \|_{\ell_r} \geq \sum_{i \in K} \| \mathcal{U}_i \|_{\ell_r} - \| \mathcal{V}_i \|_{\ell_r},\]
where the last step follows from the reverse triangle inequality for the matrix $\ell_r$-norm. Now, since $\mathsf{rowsupp}(\mathcal{U}) \subset K$ and since $\hat{x}$ achieves a lower objective in~\eqref{eq:convex_opt} than $x[0]$, $\sum_{i \in K} \| \mathcal{U}_i \|_{\ell_r} = \sum_{i=1}^n\| \mathcal{U}_i \|_{\ell_r} \geq \sum_{i=1}^n \| \mathcal{V}_i \|_{\ell_r}$. So,
\begin{align*}
    \sum_{i \in K} \| \Sha_i \|_{\ell_r} &\geq \sum_{i=1}^n \| \mathcal{V}_i \|_{\ell_r} -\sum_{i \in K} \| \mathcal{V}_i \|_{\ell_r} \\
    &= \sum_{i \in K^c} \| \mathcal{V}_i \|_{\ell_r} = \sum_{i \in K^c} \| \Sha_i \|_{\ell_r},
\end{align*}
where the last equality follows from the fact that $\mathsf{rowsupp}(\mathcal{U}) \subset K$. Hence, (2) does not hold.\qed

%\section*{Acknowledgment}

\bibliographystyle{IEEEtran}
\bibliography{IEEEabrv,mybibfile}

\begin{thebibliography}{10}
\providecommand{\url}[1]{#1}
\csname url@rmstyle\endcsname
\providecommand{\newblock}{\relax}
\providecommand{\bibinfo}[2]{#2}
\providecommand\BIBentrySTDinterwordspacing{\spaceskip=0pt\relax}
\providecommand\BIBentryALTinterwordstretchfactor{4}
\providecommand\BIBentryALTinterwordspacing{\spaceskip=\fontdimen2\font plus
\BIBentryALTinterwordstretchfactor\fontdimen3\font minus
  \fontdimen4\font\relax}
\providecommand\BIBforeignlanguage[2]{{%
\expandafter\ifx\csname l@#1\endcsname\relax
\typeout{** WARNING: IEEEtran.bst: No hyphenation pattern has been}%
\typeout{** loaded for the language `#1'. Using the pattern for}%
\typeout{** the default language instead.}%
\else
\language=\csname l@#1\endcsname
\fi
#2}}

\bibitem{moon2008chaotic}
F.~C. Moon, \emph{Chaotic and fractal dynamics: introduction for applied
  scientists and engineers}.\hskip 1em plus 0.5em minus 0.4em\relax John Wiley
  \& Sons, 2008.

\bibitem{lundstrom2008fractional}
B.~N. Lundstrom, M.~H. Higgs, W.~J. Spain, and A.~L. Fairhall, ``Fractional
  differentiation by neocortical pyramidal neurons,'' \emph{Nature
  \text{N}euroscience}, vol.~11, no.~11, p. 1335, 2008.

\bibitem{werner2010fractals}
G.~Werner, ``Fractals in the nervous system: conceptual implications for
  theoretical neuroscience,'' \emph{Frontiers in \text{P}hysiology}, vol.~1,
  p.~15, 2010.

\bibitem{turcott1996fractal}
R.~G. Turcott and M.~C. Teich, ``Fractal character of the electrocardiogram:
  distinguishing heart-failure and normal patients,'' \emph{Annals of
  \text{B}iomedical \text{E}ngineering}, vol.~24, no.~2, pp. 269--293, 1996.

\bibitem{thurner2003scaling}
S.~Thurner, C.~Windischberger, E.~Moser, P.~Walla, and M.~Barth, ``Scaling laws
  and persistence in human brain activity,'' \emph{Physica A: Statistical
  Mechanics and its Applications}, vol. 326, no. 3-4, pp. 511--521, 2003.

\bibitem{teich1997fractal}
M.~C. Teich, C.~Heneghan, S.~B. Lowen, T.~Ozaki, and E.~Kaplan, ``Fractal
  character of the neural spike train in the visual system of the cat,''
  \emph{Journal of the \text{O}ptical \text{S}ociety of \text{A}merica},
  vol.~14, no.~3, pp. 529--546, 1997.

\bibitem{chen2010anomalous}
W.~Chen, H.~Sun, X.~Zhang, and D.~Koro{\v{s}}ak, ``Anomalous diffusion modeling
  by fractal and fractional derivatives,'' \emph{Computers \& Mathematics with
  Applications}, vol.~59, no.~5, pp. 1754--1758, 2010.

\bibitem{jaishankar2012power}
A.~Jaishankar and G.~H. McKinley, ``Power-law rheology in the bulk and at the
  interface: quasi-properties and fractional constitutive equations,''
  \emph{Proceedings of the Royal Society of London A: Mathematical, Physical
  and Engineering Sciences}, vol. 469, no. 2149, 2013.

\bibitem{petravs2011fractional}
I.~Petr{\'a}{\v{s}}, ``Fractional-order chaotic systems,'' in
  \emph{Fractional-order nonlinear systems}.\hskip 1em plus 0.5em minus
  0.4em\relax Springer, 2011, pp. 103--184.

\bibitem{west2014networks}
B.~J. West, M.~Turalska, and P.~Grigolini, \emph{Networks of echoes: imitation,
  innovation and invisible leaders}.\hskip 1em plus 0.5em minus 0.4em\relax
  Springer Science \& Business Media, 2014.

\bibitem{xuecps}
Y.~Xue, S.~Rodriguez, and P.~Bogdan, ``A spatio-temporal fractal model for a
  cps approach to brain-machine-body interfaces,'' in \emph{2016 Design,
  Automation Test in Europe Conference Exhibition (DATE)}, March 2016, pp.
  642--647.

\bibitem{xue2017reliable}
Y.~Xue and P.~Bogdan, ``Reliable multi-fractal characterization of weighted
  complex networks: algorithms and implications,'' \emph{Scientific
  \text{R}eports}, vol.~7, no.~1, p. 7487, 2017.

\bibitem{magin2006fractional}
R.~L. Magin, \emph{Fractional calculus in bioengineering}.\hskip 1em plus 0.5em
  minus 0.4em\relax Begell House Redding, 2006.

\bibitem{baleanu2011fractional}
D.~Baleanu, J.~A.~T. Machado, and A.~C. Luo, \emph{Fractional dynamics and
  control}.\hskip 1em plus 0.5em minus 0.4em\relax Springer Science \& Business
  Media, 2011.

\bibitem{britton2016electroencephalography}
J.~W. Britton, L.~C. Frey, J.~Hopp, P.~Korb, M.~Koubeissi, W.~Lievens,
  E.~Pestana-Knight, and E.~L. St, \emph{Electroencephalography (EEG): An
  introductory text and atlas of normal and abnormal findings in adults,
  children, and infants}.\hskip 1em plus 0.5em minus 0.4em\relax American
  Epilepsy Society, Chicago, 2016.

\bibitem{shoukry2017secure}
Y.~Shoukry, P.~Nuzzo, A.~Puggelli, A.~L. Sangiovanni-Vincentelli, S.~A. Seshia,
  and P.~Tabuada, ``Secure state estimation for cyber-physical systems under
  sensor attacks: A satisfiability modulo theory approach,'' \emph{IEEE
  Transactions on Automatic Control}, vol.~62, no.~10, pp. 4917--4932, 2017.

\bibitem{pajic2014robustness}
M.~Pajic, J.~Weimer, N.~Bezzo, P.~Tabuada, O.~Sokolsky, I.~Lee, and G.~J.
  Pappas, ``Robustness of attack-resilient state estimators,'' in
  \emph{Proceedings of the ACM/IEEE 5th International Conference on
  Cyber-Physical Systems}.\hskip 1em plus 0.5em minus 0.4em\relax IEEE Computer
  Society, 2014, pp. 163--174.

\bibitem{pajic2017attack}
M.~Pajic, I.~Lee, and G.~J. Pappas, ``Attack-resilient state estimation for
  noisy dynamical systems,'' \emph{IEEE Transactions on Control of Network
  Systems}, vol.~4, no.~1, pp. 82--92, 2017.

\bibitem{hu2018state}
L.~Hu, Z.~Wang, Q.-L. Han, and X.~Liu, ``State estimation under false data
  injection attacks: Security analysis and system protection,''
  \emph{Automatica}, vol.~87, pp. 176--183, 2018.

\bibitem{mishra2015secure}
S.~Mishra, Y.~Shoukry, N.~Karamchandani, S.~Diggavi, and P.~Tabuada, ``Secure
  state estimation: Optimal guarantees against sensor attacks in the presence
  of noise,'' in \emph{Proceedings of the IEEE International Symposium on
  Information Theory (ISIT)}.\hskip 1em plus 0.5em minus 0.4em\relax IEEE,
  2015, pp. 2929--2933.

\bibitem{shi2017finite}
D.~Shi, R.~J. Elliott, T.~Chen, \emph{et~al.}, ``On finite-state stochastic
  modeling and secure estimation of cyber-physical systems.'' \emph{IEEE
  Transactions on Automatic Control}, vol.~62, no.~1, pp. 65--80, 2017.

\bibitem{an2018secure}
L.~An and G.-H. Yang, ``Secure state estimation against sparse sensor attacks
  with adaptive switching mechanism,'' \emph{IEEE Transactions on Automatic
  Control}, vol.~63, no.~8, pp. 2596--2603, 2018.

\bibitem{mo2015secure}
Y.~Mo and B.~Sinopoli, ``Secure estimation in the presence of integrity
  attacks,'' \emph{IEEE Transactions on Automatic Control}, vol.~60, no.~4, pp.
  1145--1151, 2015.

\bibitem{forti2017secure}
N.~Forti, G.~Battistelli, L.~Chisci, and B.~Sinopoli, ``Secure state estimation
  of cyber-physical systems under switching attacks,''
  \emph{IFAC-PapersOnLine}, vol.~50, no.~1, pp. 4979--4986, 2017.

\bibitem{pasqualetti2013attack}
F.~Pasqualetti, F.~D{\"o}rfler, and F.~Bullo, ``Attack detection and
  identification in cyber-physical systems,'' \emph{IEEE Transactions on
  Automatic Control}, vol.~58, no.~11, pp. 2715--2729, 2013.

\bibitem{fawzi2014secure}
H.~Fawzi, P.~Tabuada, and S.~Diggavi, ``Secure estimation and control for
  cyber-physical systems under adversarial attacks,'' \emph{IEEE Transactions
  on Automatic Control}, vol.~59, no.~6, pp. 1454--1467, 2014.

\bibitem{ding2018survey}
D.~Ding, Q.-L. Han, Y.~Xiang, X.~Ge, and X.-M. Zhang, ``A survey on security
  control and attack detection for industrial cyber-physical systems,''
  \emph{Neurocomputing}, vol. 275, pp. 1674--1683, 2018.

\bibitem{fawzi2011}
H.~Fawzi, P.~Tabuada, and S.~Diggavi, ``Secure state-estimation for dynamical
  systems under active adversaries,'' in \emph{Proceedings of the 49th Annual
  Allerton Conference on Communication, Control, and Computing}, Sept 2011, pp.
  337--344.

\bibitem{mo2014resilient}
Y.~Mo, J.~P. Hespanha, and B.~Sinopoli, ``Resilient detection in the presence
  of integrity attacks,'' \emph{IEEE Transactions on Signal Processing},
  vol.~62, no.~1, pp. 31--43, 2014.

\bibitem{sabatier2012observability}
J.~Sabatier, C.~Farges, M.~Merveillaut, and L.~Feneteau, ``On observability and
  pseudo state estimation of fractional order systems,'' \emph{European
  \text{J}ournal of \text{C}ontrol}, vol.~18, no.~3, pp. 260--271, 2012.

\bibitem{sierociuk2006fractional}
D.~Sierociuk and A.~Dzieli{\'n}ski, ``Fractional kalman filter algorithm for
  the states, parameters and order of fractional system estimation,''
  \emph{International Journal of Applied Mathematics and Computer Science},
  vol.~16, pp. 129--140, 2006.

\bibitem{Safari1}
B.~Safarinejadian, N.~Kianpour, and M.~Asad, ``State estimation in
  fractional-order systems with coloured measurement noise,''
  \emph{Transactions of the Institute of Measurement and Control}, vol.~40,
  no.~6, pp. 1819--1835, 2018.

\bibitem{Safari2}
B.~Safarinejadian, M.~Asad, and M.~S. Sadeghi, ``Simultaneous state estimation
  and parameter identification in linear fractional order systems using
  coloured measurement noise,'' \emph{International Journal of Control},
  vol.~89, no.~11, pp. 2277--2296, 2016.

\bibitem{miljkovic2017ecg}
N.~Miljkovi{\'c}, N.~Popovi{\'c}, O.~Djordjevi{\'c}, L.~Konstantinovi{\'c}, and
  T.~B. {\v{S}}ekara, ``Ecg artifact cancellation in surface emg signals by
  fractional order calculus application,'' \emph{Computer methods and programs
  in biomedicine}, vol. 140, pp. 259--264, 2017.

\bibitem{DzielinskiFOS}
A.~Dzielinski and D.~Sierociuk, ``Adaptive feedback control of fractional order
  discrete state-space systems,'' in \emph{Proceedings of the International
  Conference on Computational Intelligence for Modelling, Control and
  Automation and International Conference on Intelligent Agents, Web
  Technologies and Internet Commerce (CIMCA-IAWTIC'06)}, vol.~1, Nov 2005, pp.
  804--809.

\bibitem{Debener2010}
S.~Debener, C.~Kranczioch, and I.~Gutberlet, \emph{EEG Quality: Origin and
  Reduction of the EEG Cardiac-Related Artefact}, C.~Mulert and L.~Lemieux,
  Eds.\hskip 1em plus 0.5em minus 0.4em\relax Berlin, Heidelberg: Springer
  Berlin Heidelberg, 2010.

\bibitem{gupta2018}
G.~Gupta, S.~Pequito, and P.~Bogdan, ``Dealing with unknown unknowns:
  Identification and selection of minimal sensing for fractional dynamics with
  unknown inputs,'' in \emph{Proceedings of the American Control Conference},
  June 2018, pp. 2814--2820.

\bibitem{guermah2008}
S.~Guermah, S.~Djennoune, and M.~Bettayeb, ``Controllability and observability
  of linear discrete-time fractional-order systems,'' \emph{International
  Journal of Applied Mathematics and Computer Science}, vol.~18, no.~2, pp.
  213--222, 2008.

\bibitem{guru2008}
V.~Guruswami, J.~R. Lee, and A.~Wigderson, ``Euclidean sections of $\ell_1^n$
  with sublinear randomness and error-correction over the reals,'' in
  \emph{Approximation, Randomization and Combinatorial Optimization. Algorithms
  and Techniques}, A.~Goel, K.~Jansen, J.~D.~P. Rolim, and R.~Rubinfeld,
  Eds.\hskip 1em plus 0.5em minus 0.4em\relax Berlin, Heidelberg: Springer
  Berlin Heidelberg, 2008, pp. 444--454.

\bibitem{CandesTao}
E.~J. Candes and T.~Tao, ``Decoding by linear programming,'' \emph{IEEE
  Transactions on Information Theory}, vol.~51, no.~12, pp. 4203--4215, Dec
  2005.

\bibitem{davenport2011}
M.~A. Davenport, M.~F. Duarte, Y.~C. Eldar, and G.~Kutyniok, ``Introduction to
  compressed sensing,'' \emph{Compressed Sensing: Theory and Applications},
  vol.~93, no.~1, p.~2, 2011.

\bibitem{Xue2016}
Y.~Xue, S.~Pequito, J.~R. Coelho, P.~Bogdan, and G.~J. Pappas, ``Minimum number
  of sensors to ensure observability of physiological systems: A case study,''
  in \emph{Proceedings of the 54th Annual Allerton Conference on Communication,
  Control, and Computing}, Sept 2016, pp. 1181--1188.

\bibitem{cvx}
I.~CVX~Research, ``{CVX}: Matlab software for disciplined convex programming,
  version 2.0,'' \url{http://cvxr.com/cvx}, Aug. 2012.

\bibitem{grant08}
M.~Grant and S.~Boyd, ``Graph implementations for nonsmooth convex programs,''
  in \emph{Recent Advances in Learning and Control}, ser. Lecture Notes in
  Control and Information Sciences, V.~Blondel, S.~Boyd, and H.~Kimura,
  Eds.\hskip 1em plus 0.5em minus 0.4em\relax Springer-Verlag Limited, 2008,
  pp. 95--110.

\bibitem{SchalkBCI}
G.~Schalk, D.~J. McFarland, T.~Hinterberger, N.~Birbaumer, and J.~R. Wolpaw,
  ``\text{BCI}2000: a general-purpose brain-computer interface (\text{BCI})
  system,'' \emph{IEEE Transactions on Biomedical Engineering}, vol.~51, no.~6,
  pp. 1034--1043, June 2004.

\bibitem{goldberger2000physiobank}
A.~L. Goldberger, L.~A. Amaral, L.~Glass, J.~M. Hausdorff, P.~C. Ivanov, R.~G.
  Mark, J.~E. Mietus, G.~B. Moody, C.-K. Peng, and H.~E. Stanley, ``Physiobank,
  physiotoolkit, and physionet: components of a new research resource for
  complex physiologic signals,'' \emph{Circulation}, vol. 101, no.~23, pp.
  e215--e220, 2000.

\bibitem{marzbani2016neurofeedback}
H.~Marzbani, H.~R. Marateb, and M.~Mansourian, ``Neurofeedback: a comprehensive
  review on system design, methodology and clinical applications,'' \emph{Basic
  and \text{C}linical \text{N}euroscience}, vol.~7, no.~2, p. 143, 2016.

\end{thebibliography}

\end{document}